\documentclass[a4paper,11pt]{article}
\usepackage{amsthm,amsmath,amssymb,amsfonts,mathrsfs,hyperref,microtype}
\usepackage{graphicx,color}

\newtheorem{defn}{Definition}[section]
\newtheorem{thm}{Theorem}[section]
\newtheorem{prop}{Proposition}[section]
\newtheorem{cor}{Corollary}[section]
\newtheorem{rmk}{Remark}[section]

\newtheorem{Assumption}{Assumption}[section]
\newtheorem{exm}{Example}[section]

\newcommand{\ud}{\mathrm{d}}

\def\N{{\rm I\kern-0.16em N}}
\def\R{{\rm I\kern-0.16em R}}
\def\E{{\rm I\kern-0.16em E}}
\def\Pro{{\rm I\kern-0.16em P}}
\def\F{{\rm I\kern-0.16em F}}
\def\B{{\rm I\kern-0.16em B}}
\def\C{{\rm I\kern-0.46em C}}
\def\G{{\rm I\kern-0.50em G}}

\numberwithin{equation}{section}
\font\eka=cmex10
\usepackage{ae}
\def\ind{\mathrel{\hbox{\rlap{%
\hbox to 7.5pt{\hrulefill}}\raise6.6pt\hbox{\eka\char'167}}}}
\begin{document}
\title{\textbf{A general approach to small deviation via concentration of measures}}
\date{\today}




\renewcommand{\thefootnote}{\fnsymbol{footnote}}

\author{Ehsan Azmoodeh\footnotemark[1] \, and \, Lauri Viitasaari\footnotemark[2],\footnotemark[3]}

\footnotetext[2]{Department of Mathematics and System Analysis, Aalto University School of Science, Helsinki P.O. Box 11100, FIN-00076 Aalto,  Finland, {\tt lauri.viitasaari@aalto.fi}.}

\footnotetext[3]{Department of Mathematics, Saarland University, Post-fach 151150, D-66041 Saarbr\"ucken, Germany.}

\footnotetext[1]{Mathematics Research Unit, Luxembourg University, P.O. Box L-1359, Luxembourg, {\tt ehsan.azmoodeh@uni.lu}. Azmoodeh is supported by research project F1R-MTH-PUL-12PAMP.}
\maketitle

\abstract{We provide a general approach to obtain upper bounds for small deviations $ \Pro(\Vert y
\Vert \le \epsilon)$ in different norms, namely the supremum and $\beta$- H\"older norms. The large class of processes
$y$ under consideration takes the form
$y_t= X_t + \int_0^t a_s \ud s$, where $X$ and $a$ are two possibly dependent stochastic processes. Our approach provides an upper bound for small deviations whenever upper bounds for 
the \textit{concentration of measures} of $L^p$- norm of random vectors built from increments of the process $X$ and \textit{large deviation} estimates for the process $a$ are available. Using our method, among 
others, we obtain the optimal rates of small deviations in supremum and $\beta$- H\"older norms for fractional Brownian motion with Hurst parameter $H\le\ \frac{1}{2}$. As an application, we discuss the usefulness of our upper bounds for small 
deviations in pathwise stochastic integral representation of random variables motivated by
the hedging problem in mathematical finance.

\



\tableofcontents

\section{Introduction}

\subsection{Overview and motivation}
General small deviation problems have received a lot of attention recently due to their deep connections to various mathematical
topics such as operator theory, quantization, almost sure limit theorems, etc., see 
surveys \cite{l-s, Lifshits} and references therein. More recently, a link was 
established between small deviations and problems in mathematical statistics, namely functional analysis of data
and nonparametric Bayes estimates \cite{Fe-Vi, vaart-zanten, Au-Ib-Li-Za}. 

Let $y$ be a stochastic process (sequence) with sample paths lying in some
functional normed space with the norm denoted by $\Vert \, \Vert$. The general small deviation problems (or small ball probabilities)
study the asymptotic behavior of the probability $\Pro (\Vert y \Vert \le \epsilon)$ as $\epsilon \to 0$,
whereas the large deviation investigates the asymptotic behavior of the probability $\Pro(\Vert y \Vert \ge x)$ as $x \to
\infty$.\\

The small deviation problem has a long history and is considered a difficult problem in general. The main obstacles to develop a
unified approach to study small deviation problem are the adherence to the underlying stochastic process $y$ and to
the norm $\Vert \, \Vert$ under which the small ball probability is considered. Therefore, in most of the literature the small
deviation problem is usually studied for a particular class of processes and under a particular norm. It can be said that
the one of the first successful attempt to develop a general approach is due to  W. Stolz \cite{stolz1,stolz2} using the \textit{Schauder basis}. His approach covers
almost all Gaussian processes having similar covariance functions to that of \textit{fractional Brownian motion}. Another special
effort in this direction is made in \cite{lift-simon} by Lifshits \& Simon in which contains some non-Gaussian processes, in particular \textit{fractional stable} processes. Developing a general strategy to 
deal with the small deviation problem for Gaussian processes is culminated with giving a precise link, discovered by Kuelbs and Li
\cite{ku-li} and completed by Li and Linde \cite{li-li}, to the metric entropy of the unit ball of the reproducing kernel
Hilbert space generated by Gaussian process. In the non-Gaussian case, similar links are built in \cite{stable1, stable2, stable3} for \textit{symmetric} $\alpha$-\textit{stable} processes. Apparently, 
it remains a great challenge to find some principle describing small
deviations for general classes of processes and norms, rather than investigate the problem case by case.\\

\subsection{Main results and advantages}

The aim of this paper is to provide a general methodology that can be applied in both discrete and continuous setting, to give upper bounds (in fact exponential upper bounds in many interesting examples) 
for small deviations. Consider stochastic processes of the form 
\begin{equation}\label{proc-form}
y_t = X_t +\int_{0}^{t} a_s \ud s, \qquad t \in [0,T],
\end{equation}
where $X$ and $a$ are two stochastic processes, possibly dependent, such that $X_0=0$ and the Lebesgue integral in $(\ref{proc-form})$ is well defined almost surely. Let $N \in \N,$ and $p, \delta >0$. For a 
given partition $\{t_k\}_{k=0}^{N}$ of the interval $[0,T]$ such that $t_k - t_{k-1} = \delta$, we set $$\mathbf{\vert X \vert}_p
=\left[\sum_{k=1}^{N}\left|X_{t_{k}}-X_{t_{k-1}}\right|^p\right]^{\frac{1}{p}}.$$ Our main finding states that for carefully
chosen parameters $N, \delta, p$ and sufficiently small $\epsilon>0$,  we have the upper bound 

\begin{equation}\label{main-inequality}
\Pro(\Vert y \Vert_\infty \le \epsilon) \le \Pro \left(\big|\mathbf{\vert X \vert }_p-I \big|\geq c_\epsilon \right) + \Pro
\left( \Vert a\Vert_\infty \geq d_\epsilon \right),
\end{equation}

where $\Vert \, \Vert_\infty$ stands for the supremum norm, $I$ can be taken as a \textit{median} of the random variable $\mathbf{\vert X \vert}_p$, and the constants $c_\epsilon, d_\epsilon$ are such that 
$c_\epsilon \to 0$ and $d_\epsilon \to \infty$ as $\epsilon \to 0$. Hence, the probabilities appearing in the right hand side of the $(\ref{main-inequality})$ connects small ball probabilities to the
concentration of measures and to the theory of large deviation; two topics that are of great interest and have been developed extensively. As a result, the exponential upper bounds for small ball probabilities are derived 
as soon as there exist exponential upper bounds for the corresponding concentration of measure probability and the tail probability in $(\ref{main-inequality})$. It is worth to mention that this is the case 
in many interesting situations as is shown in Section \ref{sec:exam}. A notable example is the case when the process $X$ is Gaussian with some regularity assumptions on the \textit{incremental variance} function and the process $a$ is 
bounded almost surely.\\

Our approach has several important advantages compared to the classical methods. Firstly, our method works for general processes and
we do not need to assume any demanding assumptions on the underlying process $y$. 
As the second advantage, in the literature (see \cite{m-r} and \cite[Section 4.2]{l-s} for example) the small ball probabilities for Gaussian processes are mostly restricted to the class
of \textit{stationary increments} which can be considerably extended with our approach. This is the topic of the Subsection \ref{gaussian}. Moreover, it is well-known that the
estimates for the small ball probabilities for Gaussian processes is deeply connected to the incremental variance of the
process. It is pointed out in Li and Shao \cite{l-s} (see also Lifshits \cite{Lifshits}) that to obtain upper bound for small
ball probability it is not sufficient to have lower bound for incremental variance in general. However, we will show that using our method, this is exactly
the key element to obtain the exponential upper bounds (see Theorem \ref{thm:gaussian}) in many cases.\\

\subsection{Plan}
The rest of the paper is organized as follows. In Section \ref{sec:main} we formulate and proof
our main theorems. The Section \ref{sec:exam} is devoted to examples where we study three particular cases; i) sums of independent random variables, ii) H\"older continuous processes with bounded H\"older constant, and iii) a wide class of Gaussian processes. In Section \ref{application} we consider the usefulness of our exponential upper bounds for
small deviations in stochastic integral representation of random variables.


\section{Main results: general approach}\label{sec:main}

In what follows, all random objects are defined on a complete probability space $(\Omega,\mathcal{F},\Pro)$ .

\subsection{Small deviation in the supremum norm}

We consider the stochastic processes of the form $(T>0)$

\begin{equation}\label{y-process}
 y_t = X_t + \int_0^t a_s \ud s \quad t \in [0,T].
\end{equation}

Here $X=\{X_t\}_{t \in [0,T]}$ with $X_0=0$ and $a =\{a_t\}_{t \in [0,T]}$ are any general stochastic processes such that the
Lebesgue integral is well-defined. Notice that we do not assume that the two processes $X$ and $a$ are independent, otherwise a standard application of the Anderson's 
inequality \cite[Theorem 2.13]{l-s} implies 
that an upper bound for small ball probability for the process $y$ reduces to obtain an upper bound for small ball probability of the process $X$. Therefore, the independence is not a graceful assumption for 
our purposes in this paper. Moreover, in general $X$ and $a$ are not assumed to be continuous. We define
the supremum norm on $[0,T]$ by $\Vert y\Vert_{\infty} = \sup_{t\in[0,T]}|y_t|$. For further use, we set $\mathbf{X}=(X_{t_0},
\cdots,X_{t_N})$ for a given sequence of time points $\{0=t_0<t_1<\ldots<t_N=T\}$ with
$t_{k+1}-t_k=\delta$. We consider different $L^p$-norms, and we set

\begin{equation*}
\mathbf{\vert X \vert}_p =\left[\sum_{k=1}^{N}\left|X_{t_{k}}-X_{t_{k-1}}\right|^p\right]^{\frac{1}{p}}.
\end{equation*}

For given $\epsilon>0$, we also define the following set:

\begin{equation*}
\mathcal{A}_p (\epsilon) = \left\{ (N,\delta,I) \in \mathbb{N} \times \R_+ \times \R_+ \big \vert  4N^{\frac{1}{p}}\leq I\epsilon^{-1} \text{ and } N\delta \leq T \right\},
\end{equation*}
where $T$ denotes the length of the time interval under consideration, and $\R_+ = (0,\infty)$. Hereafter, without ambiguity we will drop the dependency of the set $\mathcal{A}_p$ on the parameter $\epsilon$, 
and we write $\mathcal{A}_p(\epsilon) = \mathcal{A}_p$. The set $\mathcal{A}_p$ includes all required parameters which relates them in a way that are acceptable for our purposes. Note that the set 
$\mathcal{A}_p$ is never empty. \\

The following theorem explains our general approach how small ball probabilities can be related to concentration of measure phenomena of the process $X$ and large deviation of the process $a$.

\begin{thm}\label{general-thm}
Let the above notations and assumptions prevail. Then for any $\epsilon>0$ and for any interval $[0,T]$, we have
\begin{equation*}
\begin{split}
\Pro(\Vert y\Vert_{\infty}\leq \epsilon) & \leq \inf_{p>0}\inf_{(N,\delta,I)\in\mathcal{A}_p}\Big\{ \Pro \left(\big|
\mathbf{\vert X \vert }_p-I \big|\geq 2^{-2}I\right)\\
& \quad \hspace{0.2cm} +\Pro\left( \Vert a\Vert_{\infty} \geq
2^{-2}IN^{-\frac{1}{p}}\delta^{-1}\right)\Big\}.
\end{split}
\end{equation*}
\end{thm}

\begin{rmk}{\rm
The given upper bound may look rather odd at first sight and it is not clear immediately how the upper bound is related to concentration of measure and large deviation for the process $a$. We will explain 
the relation clearly in Subsection \ref{subsec:concentration} and it will come particularly clear in Section \ref{sec:exam} where we consider examples.
}
\end{rmk}

\begin{proof}[Proof of Theorem \ref{general-thm}]
For simplicity, we choose $p=2$. However, the general case follows by substituting
$\mathbf{\vert X\vert}_2$ with  $\mathbf{\vert X\vert}_p$ and $\sqrt{N}$ with $N^{\frac{1}{p}}$. For any $t,s\in[0,T]$, we have

\begin{equation*}
 |X_t-X_s| \leq 2\|y\|_{\infty}+|t-s|  \Vert a \Vert_{\infty}.
\end{equation*}
Let now the vector $(N, \delta,I)\in \mathcal{A}_2$ be fixed. Consider time points $\{t_k\}_{k=0}^{N}$ such that $t_0=0$,
$t_N=T$ and $t_{k}-t_{k-1}=\delta$. For such partition, we get

\begin{equation*}
\left| X_{t_{k}}-X_{t_{k-1}}\right|\leq 2 \|y\|_{\infty}+\delta \Vert a\Vert_{\infty}.
\end{equation*}

By taking squares on both sides and summing up, we obtain

\begin{equation*}
\sum_{k=1,\ldots, N}\left|X_{t_{k}}-X_{t_{k-1}}\right|^2\leq N \Big( 2 \|y\|_{\infty}+\delta\|a\|_{\infty}\Big)^2.
\end{equation*}

We take square roots to obtain that

\begin{equation}
\label{eq:aux}
\mathbf{\vert X \vert }_2 \leq 2\sqrt{N}\|y\|_{\infty}+\sqrt{N}\delta\|a\|_{\infty}.
\end{equation}
For the positive number $I$, the triangle inequality gives
$$
I \leq \mathbf{\vert X \vert }_2 + |\mathbf{\vert X \vert }_2-I|
$$
Hence multiplying both sides with $I^{-1}$ yields
$$
1 \leq \frac{\mathbf{\vert X \vert }_2}{I} + \frac{|\mathbf{\vert X \vert }_2-I|}{I}.
$$
Combining with (\ref{eq:aux}), we obtain for any $\epsilon>0$ that
\begin{equation*}
\begin{split}
1 & \leq  2\frac{\sqrt{N}}{I}\|y\|_{\infty}+\frac{\sqrt{N}\delta}{I}\|a\|_{\infty}+\frac{|\mathbf{\vert X \vert }_2-I|}{I}\\
&=2\frac{\sqrt{N}}{I}\|y\|_{\infty}+\frac{\sqrt{N}\delta}{I}\|a\|_{\infty}+\frac{|\mathbf{\vert X \vert }_2-I|}{I}.
\end{split}
\end{equation*}

Now on the set $\mathcal{A}_2$, we have $\frac{2\sqrt{N}}{I}\leq \frac{1}{2\epsilon}$. Hence
\begin{equation}
\label{eq:aux2}
 1 \leq \frac{\|y\|_{\infty}}{2\epsilon} + \frac{\sqrt{N}\delta}{I}\|a\|_{\infty}+ \frac{\big\vert \mathbf{\vert X\vert}_2-I\big\vert}{I}.
\end{equation}
Applying (\ref{eq:aux2}) on the set $\{\|y\|_{\infty} \le \epsilon \}$, we also obtain
$$
\frac{\sqrt{N}\delta}{I}\|a\|_{\infty}+\frac{|\mathbf{\vert X \vert }_2-I|}{I} \geq \frac12.
$$
It remains to note that for any positive random variables $Z_1$ and $Z_2$ and any number $a>0$, we have the inequality
$$
\Pro(Z_1 + Z_2 >a) \leq \Pro\left(Z_1 > \frac{a}{2}\right)+\Pro\left(Z_2 > \frac{a}{2}\right).
$$
Consequently, we get
\begin{equation*}
\begin{split}
&\Pro \left(\|y\|_{\infty} \le \epsilon\right)\\
 & \leq   \Pro\left(\frac{\big\vert \mathbf{\vert X\vert}_2-I\big\vert}{I} \geq 2^{-2}\right)+\Pro\left(\frac{\sqrt{N}\delta}{I}\|a\|_ {\infty} \geq 2^{-2}\right)\\
&= \Pro\left( \big\vert \mathbf{\vert X\vert}_2-I\big\vert \geq 2^{-2}I\right)+\Pro\left( \Vert a\Vert_{\infty} \geq 2^{-2}IN^{-\frac{1}{2}}\delta^{-1}\right).
\end{split}
\end{equation*}
Now this upper bound holds for any numbers $(N,\delta,I)\in\mathcal{A}_2$ while the left side is independent of these parameters. Hence, we obtain the result by taking infinitum.
\end{proof}
\begin{rmk}{\rm
Note that in the case $a=0$, the term $2^{-2}$ can be replaced with $2^{-1}$ on the probability $\Pro\left( \big\vert \mathbf{\vert X\vert}_2-I\big\vert \geq 2^{-2}I\right)$. However, this essentially affects only to the constants.
}
\end{rmk}
\begin{rmk}{\rm
For simplicity we chose uniform division of time points $\{t_k; \, k=1,\cdots,N\}$. However, by examining the above proof, it is clear that one can formulate the result for non-uniform partitions by 
replacing $\delta$ with $\delta=\max_{1 \le k \le N}\vert t_{k-1}-t_k \vert$, i.e. the mesh of the partition. In this case, one may take infimum with respect to partitions rather than parameters $N$ and $\delta$.
}
\end{rmk}

\subsection{Small deviation in other norms}
To demonstrate the power of our general methodology, we devote this section to small deviation in other norms, in particular
$L^1$- and $\beta$- H\"older norms for $\beta \in (0,1)$. We recall that for any measurable function $f:[0,T] \to \R$, the $L^1$ and
$\beta$- H\"older norms are defined as following:

$$
\Vert f\Vert_{L^1} = \int_0^T |f(u)|\ud u \quad \text{and} \quad \Vert f\Vert_\beta = \sup_{0\leq s\neq t\leq T}
\frac{|f(t)-f(s)|}{(t-s)^\beta}.
$$

For given $\epsilon>0$, we will consider the following sets:
\begin{equation*}
\mathcal{\widehat{A}}_p (\epsilon) = \left\{ (N,\delta,I) \in \N \times \R_+ \times \R_+  \big \vert
 \, 8N^{\frac{1}{p}}\leq I\epsilon^{-1} \text{ and } N\delta \leq T \right\},
\end{equation*}
and
\begin{equation*}
\mathcal{\widetilde{A}}_p (\epsilon) = \left\{ (N,\delta,I) \in \N \times \R_+ \times \R_+ \big \vert
 \, 2 \delta^\beta N^{\frac{1}{p}}\leq I\epsilon^{-1} \text{ and } N\delta \leq T \right\}.
\end{equation*}

The first result shows the possibility of replacing the supremum norm with $L^1$ norm for the process $a$ in the
second probability appearing in the upper bound in Theorem \ref{general-thm}.

\begin{thm}\label{general-thm2}
Assume that all the above notations and assumptions prevail. Then for any $\epsilon>0$ and for any interval $[0,T]$, we have
\begin{equation*}
\begin{split}
\Pro(\Vert y\Vert_{\infty}\leq \epsilon) &\leq \inf_{p\geq 1}\inf_{(N,\delta,I)\in\mathcal{\widehat{A}}_p}\Big\{\Pro
\left(\big| \mathbf{\vert X \vert }_p-I \big|\geq 2^{-2}I\right)\\
& \qquad \qquad +\Pro\left( \Vert a\Vert_{L^1} \geq 2^{-3}I\right)\Big\}.
\end{split}
\end{equation*}
\end{thm}
\begin{proof}
In the proof of Theorem \ref{general-thm}, we use the bound $\left|\int_s^t a_u \ud u\right| \leq |t-s|\Vert a\Vert_{\infty}$
which is rather large upper bound. Instead, for time points $\{t_k, k=1,\ldots, N \}$,  we can write
$$
|X_{t_k}-X_{t_{k-1}}| \leq 2\Vert y \Vert_{\infty} + \int_{t_{k-1}}^{t_k}|a_u|\ud u.
$$
This leads to
$$
\sum_{k=1}^{N}|X_{t_k}-X_{t_{k-1}}|^p \leq \sum_{k=1}^{N}\left(2\Vert y \Vert_{\infty} + \int_{t_{k-1}}^{t_k}|a_u|\ud
u\right)^p.
$$
Now using the elementary inequality $(a+b)^p \leq 2^p(a^p + b^p), \, \forall \, a, b \ge 0$, we obtain the upper bound
\begin{equation}
\label{eq:aux3}
\sum_{k=1}^{N}|X_{t_k}-X_{t_{k-1}}|^p \leq N2^{2p}\Vert y\Vert_{\infty}^p +
2^p\sum_{k=1}^{N}\left(\int_{t_{k-1}}^{t_k}|a_u|\ud u\right)^p.
\end{equation}
Now using the simple fact
$$
\sum_{k=1}^{N}\left(\int_{t_{k-1}}^{t_k}|a_u|\ud u\right)^p \leq \left(\sum_{k=1}^{N}\int_{t_{k-1}}^{t_k}|a_u|\ud
u\right)^p
= \Vert a\Vert_{L^1}^p
$$
and taking power $\frac{1}{p}$ on both sides of (\ref{eq:aux3}) together with the elementary inequality $(a+b)^{\frac{1}{p}}
\leq a^{\frac{1}{p}}+b^{\frac{1}{p}}, \, p\geq  1$, we finally arrive to
$$
\mathbf{\vert X \vert }_p \leq 4N\Vert y\Vert_{\infty} + 2\Vert a \Vert_{L^1}.
$$
Now, the rest of the proof goes in the same lines as the proof of Theorem \ref{general-thm}.
\end{proof}

The next Theorem studies the small deviation of the process $y$ in the $\beta$- H\"older norm.

\begin{thm}\label{general-thm3}
Assume that all the above notations and assumptions prevail. Then for any $\epsilon>0$ and for any interval $[0,T]$, we have
\begin{equation*}
\begin{split}
\Pro(\Vert y\Vert_{\beta}\leq \epsilon) & \leq \inf_{p>0}\inf_{(N,\delta,I)\in\mathcal{\widetilde{A}}_p}\Big\{ \Pro \left(\big|
\mathbf{\vert X \vert }_p-I \big|\geq 2^{-2}I\right)\\
& \qquad \qquad +\Pro\left( \Vert a\Vert_{\infty} \geq 2^{-2}IN^{-\frac{1}{p}}\delta^{-1}\right)\Big\}.
\end{split}
\end{equation*}
\end{thm}
\begin{proof}
The starting point of the proof of Theorem \ref{general-thm} yields the inequality

\begin{equation*}
\frac{|X_t-X_s|}{\vert t-s \vert^\beta} \leq \Vert y\Vert_{\beta} + \vert t-s \vert^{1-\beta}\Vert a\Vert_{\infty}.
\end{equation*}
Therefore, for the time points $\{t_k, k=1,\ldots, N \}$ with $t_{k}-t_{k-1}=\delta$, we obtain
$$
|X_{t_k}-X_{t_{k-1}}| \leq \delta^\beta \Vert y\Vert_{\beta} + \delta\Vert a\Vert_{\infty}.
$$
Now the rest of the proof goes in the same lines as the proof of Theorem \ref{general-thm}.
\end{proof}

\subsection{Case of  Banach--valued random variables}
This subsection is devoted to the following most general version of our main findings concerning general Banach--valued random variables. For simplicity, we assume $a=0$. Let 
$(\mathcal{B},\Vert \cdot\Vert_{\mathcal{B}})$ be a Banach space and let $X$ be a Banach--valued random variable, i.e. a measurable map $X:(\Pro,\mathcal{F},\Omega) \rightarrow \mathcal{B}$. Let $N$ be a 
fixed integer. Let $\mathcal{L}=\{L_0,L_2,\ldots,L_N\}$ be any collection from the dual space $\mathcal{B}'$, i.e. any collection of linear functionals such that 
$\Vert L_k\Vert^{'}\leq 1$ for all $k=0,\cdots, N$, where $ \Vert L_k\Vert^{'}:= \sup \{ \vert L(x) \vert  : \, x \in \mathcal{B}, \, \Vert x \Vert_{\mathcal{B}} \le 1 \}$. For any $p >0$, we define
$$
\mathbf{\vert X \vert }_{\mathcal{L},p} = \left[\sum_{k=1}^N |L_{k}(X) - L_{k-1}(X)|^p\right]^\frac{1}{p}.
$$
Similarly, as before, we define the set 
$$
\mathcal{C}_p (\epsilon) = \left\{ (N,I) \in \N \times \R_+  \big \vert \,   4N^{\frac{1}{p}}\leq I\epsilon^{-1} \right\}.
$$
By recalling that a norm of a vector $X$ in Banach space $\mathcal{B}$ can be given by
$$
\Vert X\Vert_{\mathcal{B}} = \sup_{L\in \mathcal{B}': \Vert L\Vert \leq 1} |L(X)|,
$$
we obtain immediately the following result by using our approach.
\begin{thm}
Let $(\mathcal{B},\Vert \cdot\Vert_{\mathcal{B}})$ be a Banach space and let $X$ be a Banach--valued random variable. Then for any $\epsilon > 0$, we have
\begin{equation*}
\begin{split}
\Pro(\Vert X\Vert_{\mathcal{B}}\leq \epsilon) & \leq \inf_{p>0}\inf_{(N,I)\in \mathcal{C}_p}\inf_{\mathcal{L}}\Big\{ \Pro \left(\big|
\mathbf{\vert X \vert }_{\mathcal{L},p}-I \big|\geq 2^{-1}I\right) \Big\},
\end{split}
\end{equation*}
where the last infimum is taken over all collection $\mathcal{L} \subset \mathcal{B}'$ of cardinality $N$.
\end{thm}


\subsection{Relation to concentration of measures}
\label{subsec:concentration}
In this subsection, we briefly discuss the relation of our general approach to concentration of measures phenomena and explain clearly the heuristics given in the introduction. That is, how our approach to 
small deviation is related to the concentration of measures and large deviations. \\

Let $p\geq 1$ be fixed. By Theorem \ref{general-thm} we have
\begin{equation}\label{relation-concentration}
\begin{split}
\Pro(\Vert y\Vert_{\infty}\leq \epsilon) & \leq \Big\{ \Pro \left(\big|
\mathbf{\vert X \vert }_p-I \big|\geq 2^{-2}I\right)\\
& \quad \hspace{0.2cm} +\Pro\left( \Vert a\Vert_{\infty} \geq
2^{-2}IN^{-\frac{1}{p}}\delta^{-1}\right)\Big\}.
\end{split}
\end{equation}
for any $(N,\delta,I)\in\mathcal{A}_p$. Now by choosing $I = \E  \mathbf{\vert X\vert}_p$, if the expectation exists, (or more generally, $I = \mathbb{M} \mathbf{\vert X\vert}_p$ where $\mathbb{M}$ denotes the median of a 
random variable) the first term gives probability
$$
 \Pro \left(\big|
\mathbf{\vert X \vert }_p-\E \big\vert \mathbf{\vert X\vert}_p \big|\geq 2^{-2}\E  \mathbf{\vert X\vert}_p\right)
$$
which is exactly the concentration inequality for the random variable $ \mathbf{\vert X\vert}_p$ consisting of increments of process $X$. Hence, it remains to choose parameters $N$ and $\delta$ such that 
$(N,\delta,I)\in\mathcal{A}_p$ and that the probability 
$$
 \Pro \left(\big|
\mathbf{\vert X \vert }_p-\E \big\vert \mathbf{\vert X\vert}_p \big|\geq 2^{-2}\E  \mathbf{\vert X\vert}_p\right)
$$
is minimized. For the moment, assume that we can derive a lower bound for the incremental expectation, i.e. one has
$$
\E|X_{t_k}-X_{t_{k-1}}| \geq c|t_k-t_{k-1}|^\beta
$$
for some constant $c>0$ and $\beta\in(0,1)$. In Section \ref{sec:exam}, we will show that this is the case in many interesting situations, such as fractional processes. Moreover, Minkowski's inequality for 
integrals implies
$$
\E \mathbf{\vert X\vert}_p \geq \left[\sum_{k=1}^N \left(\E|\Delta_k X|\right)^p\right]^{\frac{1}{p}},
$$
where $\Delta_k X =X_{t_k}-X_{t_{k-1}}$. Hence, in this situation we obtain the lower bound $I \geq cN^{\frac{1}{p}}\delta^\beta$ for some constant $c>0$. Consequently, one can choose 
$\delta  \approx c\epsilon^{\frac{1}{\beta}}$. As a result of this choice, one may obtain an exponential upper bound for the first probability. Surprisingly, these choices lead to the known \textit{optimal rates} 
for small deviations of fractional Gaussian processes with stationary increments having bounded \textit{spectral density} function. A typical example is \textit{fractional Brownian motion} with Hurst parameter $H < \frac{1}{2}$. We 
will come back in details to this fact in Section \ref{sec:exam}. On the other hand, with these choices the second probability in the upper bound can be estimated as
\begin{equation*}
\begin{split}
\Pro\left( \Vert a\Vert_{\infty} \geq
2^{-2}IN^{-\frac{1}{p}}\delta^{-1}\right)
\leq\Pro\left( \Vert a\Vert_{\infty} \geq
c\epsilon^{\frac{\beta-1}{\beta}}\right)
\end{split}
\end{equation*}
which leads to a large deviation of the process $a$ since the exponent $\frac{\beta - 1}{\beta}$ of $\epsilon$ is negative. These observations demonstrate that an appropriate upper bound 
for small ball probabilities for sufficiently small $\epsilon$ is linked to obtaining a
"good" upper bounds for the concentration probability of the random variable $\mathbf{\vert X\vert}_p$ and the large deviation probability of the process $a$, hence justifying the claim in the introduction. 
In other words, our general methodology to obtain upper 
bounds for small deviations is linked to two well extensively studied domains in literature. For excellent references on measure concentration \& large deviations, we refer the reader to \cite{ledoux, b-l-m, dembo-zeituni}.


\section{Examples}\label{sec:exam}
In this section we show the power of our general approach by applying it to two different cases where concentration inequalities are well-known; sum of independent random variables and Gaussian processes. 
In particular, we show that by applying our method we can reproduce the optimal rates in these cases.

\subsection{Sum of independent random variables}
We begin with a naive example when $S_n=\sum_{k=1}^{n}Z_k$, for $n \ge 1$ and $\{Z_k\}_{k \ge 1}$ is a sequence of
independent and identically distributed random variables such that there exist real numbers $ c < 0 < d$ with $c \le Z_k \le d $
with probability one for every $k$. We assume $a=0$, and for convenience we set $S_0=0$. Now in this case we have
$
\mathbf{\vert S \vert}_1 = \sum_{k=1}^n |Z_k|,
$
and with choices $p=\delta =1$ and $I=\sum_{k=1}^n \E|Z_k| =n \E\vert Z_1 \vert$, Theorem \ref{general-thm} yields an upper bound
$$
\Pro(\max_{0 \le k \le n} \vert S_k \vert \le \epsilon) \le \Pro\left( \Big \vert \sum_{k=1}^n [|Z_k| - \E|Z_k|] \Big\vert \geq 2^{-1}\sum_{k=1}^n \E|Z_k|\right)
$$
for any $\epsilon \leq \frac{\E\vert Z_1\vert}{4}$.
Now an application of the \textit{Hoeffding's inequality} \cite[Theorem 2.8]{b-l-m} yields the upper bound 
\begin{equation}
\label{eq:upper_bound}
\Pro(\max_{0 \le k \le n} \vert S_k \vert \le \epsilon) \le 2\exp \big\{- \frac{n(\E\vert Z_1\vert)^2}{ 4(|c| \vee |d|)^2} \big \}.
\end{equation}
Note that for fixed $n$ the upper bound \eqref{eq:upper_bound} is not very applicable. Indeed, one would expect that as $\epsilon$ decreases to zero, 
then upper bound \eqref{eq:upper_bound} decreases to zero. However, usually one is interested in the asymptotics of such probabilities as $n$ tends to infinity. 
For this purpose the upper bound \eqref{eq:upper_bound} provides the following result.
\begin{prop}
\label{prop:iid_sum}
Let all the notations and assumptions above prevail. Let $\epsilon_n > 0$ such that $\epsilon_n \to 0$ and $\sqrt{n}\epsilon_n \leq \frac{\E\vert Z_1\vert}{4}$ for all $n\ge1$. Then
\begin{equation}\label{sum-iid}
\begin{split}
\Pro\left(\frac{1}{\sqrt{n}}\max_{0 \le k \le n} \vert S_k \vert \le \epsilon_n\right) \le 2\exp \big\{- n\frac{  (\E \vert Z_1\vert)^2}{ 4(|c| \vee |d|)^2} \big \}. 
\end{split}
\end{equation}
In particular, when $\epsilon_n = \frac{C_1}{\sqrt{n}}$ for some constant $C_1 \in\left(0, \frac{\E|Z_1|}{4}\right)$, we obtain the estimate 
$$
\Pro\left(\frac{1}{\sqrt{n}}\max_{0 \le k \le n} \vert S_k \vert \le \epsilon_n\right) \le 2\exp\{-C\, \epsilon^{-2}_n\},
$$
where $C>0$ is a constant independent of $\epsilon_n$ and $n$.
\end{prop}
\begin{rmk}{ \rm
We remark that the same analysis can be used to obtain similar bounds in the case where random variables $Z_k$ are not identically distributed. Indeed, if there exists real numbers 
$c_k$ and $d_k$ such that $c_k\leq Z_k \leq d_k$ almost surely, then the above computations leads to upper bound
$$
\Pro\left(\frac{1}{\sqrt{n}}\max_{0 \le k \le n} \vert S_k \vert \le \epsilon_n\right) \le 2\exp \big\{- n^2\frac{  \left(\frac{1}{n}\sum_{k=1}^n\E \vert Z_k\vert\right)^2}{4\sum_{k=1}^{n} (|c_k| \vee |d_k|)^2} \big \}. 
$$
for $\sqrt{n}\epsilon_n \leq \frac{1}{4 n}\sum_{k=1}^n \E|Z_k|$.
}
\end{rmk}

To compare our upper bound \eqref{sum-iid} with the existing results, it is a well--known fact that (see for
example \cite{dembo-zeituni}) under the conditions $\epsilon_n \to 0$ and $\sqrt{n} \epsilon_n
\to \infty$, when $\E(Z_1)=0$, $\E(Z^2_1)=1$ and the \textit{Cram\'er condition}, i.e. $\E \exp \{ h \vert Z_1 \vert \} < \infty$ for
some $h >0$, we have

\begin{equation*}
\Pro ( \frac{1}{\sqrt{n}} \max_{0 \le k \le n} \vert S_k \vert \le \epsilon_n ) \sim \exp \{ -
\frac{\pi^2}{8} \epsilon^{-2}_n \}.
\end{equation*}
With our method, the sequence $\epsilon_n$ is of order $n^{-\frac12}$. Hence to obtain the exponential upper bound with the optimal rate $\epsilon^{-2}_n$, in fact the assumption 
$\sqrt{n}\epsilon_n \to \infty$ can be dropped. Furthermore, our result covers case where random variables $Z_k$ are not necessarily identically distributed. On the other hand, we assumed the restriction that 
the random variables $Z_k$ are bounded almost surely. For similar and related results on behavior of the maximum of partial sums of independent and identically distributed
random variables under \textit{Berry-Esseen's} type conditions involving third moments, see the references
\cite{chung,sum-estimate1}.

\subsection{H\"older continuous processes with independent increments}

This subsection is devoted to the case when the process $X$ has independent increments and is H\"older continuous with H\"older exponent $H \in (0,1)$. 
We also assume that for some $\beta \geq H$, and for some constant $c>0$, we have 
\begin{equation}
\label{bounded_inc_exp}
\E|X_{t_k} - X_{t_{k-1}}| \geq c\delta^{\beta},
\end{equation}
where $\{t_k, k=1, \cdots,,N\}$ is any partition of the interval $[0,T]$ such that $\delta = t_k - t_{k-1}$. 

\begin{thm}\label{ex:bounded}
Assume that $X$ is a H\"older continuous process with H\"older exponent $H\in(0,1)$ such that the H\"older constant $C(\omega)$ is almost surely bounded, and that $X$ has independent increments. Furthermore, assume that (\ref{bounded_inc_exp}) holds for some $\beta\geq H$. Then for 
any $\epsilon\in(0,1)$ and any interval $[0,T]$, we have

\begin{equation*}
\Pro(\Vert X\Vert_{\infty}\leq \epsilon)  \leq 2\exp\left\{ -CT\epsilon^{-\gamma}\right\},
\end{equation*}
for some constant $C>0$, where $\gamma = \frac{1+2H-2\beta}{\beta}$.
\end{thm}
\begin{rmk}{\rm
Note that it may happen that the given upper bound does not converge to zero as $\epsilon \to 0$.  Indeed, 
if the process behaves too badly in a sense that $\beta \geq H+\frac12$, then the upper bound becomes useless. This reveals how the lower bound for expectation of increments affects to the upper bound of 
small ball probability. Note also that in many interesting examples the exponent $\beta$ in the lower bound for the increments is ''close'' to H\"older exponent $H$. 
}
\end{rmk}

\begin{proof}[Proof of Theorem \ref{ex:bounded}]
Using the H\"older continuity assumption, we obtain 
$$
|X_{t_k}-X_{t_{k-1}}|  \leq C(\omega)\delta^{H}. 
$$
Furthermore, (\ref{bounded_inc_exp}) implies $\E\mathbf{\vert X\vert}_1 \geq c N\delta^{\beta}$ and consequently, the results of previous subsection gives
$$
\Pro\left(\mathbf{|\vert X\vert}_1 - \E\mathbf{\vert X\vert}_1| > 2^{-1}\E\mathbf{\vert X\vert}_1\right) \leq 2 \exp\left(-C\frac{N^2\delta^{2\beta}}{\delta^{2H}N}\right).
$$
Now we obtain the given upper bound by choosing
$N \approx \frac{T}{\delta}$ and $\delta \approx 2^{\frac{1}{\beta}}\epsilon^{\frac{1}{\beta}}$. 
\end{proof}
\begin{rmk}{\rm We point out that up to our knowledge the exponential upper bound given in Theorem \ref{ex:bounded} for small ball probabilities for general independent increment H\"older continuous process is a new 
 result. 
 }
\end{rmk}

\subsection{Gaussian processes}\label{gaussian}

In this subsection we continue examples in continuous setup. To illustrate the power of our methodology, we derive some upper bound for small ball probabilities of the process $y$ when the process $X$ belongs 
to the class $\mathcal{X}^{(H,\beta)}$ (see Definition \ref{defn:Xalpha} below) of Gaussian processes. We stress that this class of Gaussian processes is considerably large and in particular includes the class 
of Gaussian processes with stationary increments property having H\"older continuous sample paths. It can be said that the  class of Gaussian processes with stationary increments property is the 
widest class of Gaussian processes in the literature in which the small deviation problem is considered.

We begin by recalling the following concentration inequality for $n$-dimensional Gaussian measure. 
\begin{thm}
\label{thm:gauss_con}
Let $\gamma^n$ be an $n$-dimensional Gaussian measure and let $f:\R^n \rightarrow \R$ be an $L$-Lipschitz function. Then for any $h>0$ we have
$$
\gamma^n\left(\vert f - \E f\vert \geq h\right) \leq \exp\left(-\frac{h^2}{2L^2}\right).
$$
\end{thm}
This gives us immediately the following general well-known result which is the key for our analysis. For the sake of completeness and simplicity we reproduce the proof.

\begin{cor}
\label{cor:gaussian_general_con}
Let $Y=(Y_1,Y_2,\ldots, Y_n)$ be a Gaussian vector with covariance matrix $\Gamma =(\Gamma_{ij})_{ 1 \le i,j \le n}$. Then for any $h>0$
$$
\Pro \left(\big|
\Vert Y \Vert_2-\E \Vert Y \Vert_2 \big|\geq h\right) \leq C\exp\left(-\frac{h^2}{2\Vert \Gamma\Vert_2}\right).
$$
\end{cor}
\begin{proof}
Since $\Gamma$ is positive-definite, it admits a square root matrix $A$ with elements $a_{ij}$ satisfying $A^2=\Gamma$. Moreover, $A$ is symmetric. Let now $\xi=(\xi_1,\ldots,\xi_n)$ be a vector of independent standard normal random variables. Then the vector $\tilde{Y} = A\xi$ has the same distribution as $Y$. Indeed, we have
$$
\tilde{Y}_i = \sum_{i=1}^n a_{ik}\xi_k.
$$
Hence for any $i$ and $j$ we have
$$
\E[\tilde{Y}_i\tilde{Y}_j] = \sum_{k=1}^n \sum_{p=1}^n a_{ik}a_{jp}\E[\xi_k\xi_p] = \sum_{k=1}^n a_{ik}a_{jk}.
$$
On the other hand,
$$
(A^2)_{ij} = \sum_{k=1}^n a_{ik}a_{kj} = \sum_{k=1}^n a_{ik}a_{jk}
$$
by symmetry of $A$. Define now a function $f(x_1,\ldots, x_n) = \Vert Ax\Vert_2$. This function is $\Vert A\Vert_2$-Lipschitz, and consequently Theorem \ref{thm:gauss_con} implies that
\begin{equation}
\Pro \left( \big|\Vert Y\Vert_2 - \E \Vert Y \Vert_2 \big| \geq h\right) \leq \exp\left(-C\frac{h^2}{\Vert A\Vert_2^2}\right).
\end{equation}
Next recall that the norm $\Vert A\Vert_2$ of a matrix $A$ corresponds to the largest singular value or, in the case of symmetric matrices, largest eigenvalue of $A$. On the other hand, the eigenvalues of $A^2$ are the eigenvalues of $A$ squared which gives immediately that $\Vert A\Vert_2^2 = \Vert A^2 \Vert_2 = \Vert \Gamma\Vert_2$. This completes the proof.
\end{proof}
\begin{rmk}{\rm
If we replace $\E \Vert Y \Vert_2 $ with $$\sqrt{\E \Vert Y \Vert_2^2} =\sqrt{\sum_{k=1}^N \E \left(X_{t_k}-X_{t_{k-1}}\right)^2},$$ then we have (see \cite{b-h})
$$
\Pro \left(\big|
\Vert Y \Vert_2-\sqrt{\E \Vert Y \Vert_2^2} \big|\geq h\right) \leq C\exp\left(-\frac{h^2}{4\Vert \Gamma\Vert_2}\right).
$$
This will be used throughout this section.
}\end{rmk} 
\begin{rmk}{\rm
The Corollary \ref{cor:gaussian_general_con} reveals that in the Gaussian setup, it is convenient to choose $p=2$. Indeed, for $p\neq 2$ one has to control different norm of a matrix $A$ which does not 
translate to the same norm of the covariance matrix $\Gamma$. 
}\end{rmk}

\begin{rmk}{\rm
We also remark that Corollary \ref{cor:gaussian_general_con} could also have been deduced from the concentration inequality for Banach-valued Gaussian processes which states that for any Banach--valued 
Gaussian process $X$, and any $h>0$:
$$
\Pro\left(\big| \Vert X\Vert_{\mathcal{B}} - \E\Vert X\Vert_{\mathcal{B}}   \big|  \geq h\right) \leq \exp\left(-\frac{h^2}{2\sigma^2}\right),
$$
where $\sigma^2 := \sup_{ \{ L\in \mathcal{B}':\Vert L\Vert' \leq 1 \} } \E L(X)^2$. Similarly, this result can be used for values $p\neq 2$. Indeed, for our purposes it is sufficient to consider spaces $\R^N$ 
equipped with the norm $\Vert \cdot \Vert_p$.
}
\end{rmk}

The clear message of Corollary \ref{cor:gaussian_general_con} is that to obtain an upper bound for small deviations of Gaussian process $X$, it is sufficient to control the norm of the covariance matrix 
$\Gamma$ of the increments $\Delta_k X:= X_{k \delta} - X_{(k-1) \delta}$, or equivalently, to bound the largest eigenvalue $\lambda_{\text{max}}$ of covariance matrix $\Gamma$. In particular, we wish to 
choose $\delta$ small and $N$ large which translates the problem into the controlling of the norm of the matrix asymptotically as the size of the matrix increases to infinity.\\ 

We will now turn to study the problem how to control the norm of the covariance matrix. We will begin by providing some rough bounds in a general setup. However, even these bounds turn out to provide the 
optimal rates in some cases. We will consider the following class.
\begin{defn}\label{defn:Xalpha}
Let $H\in(0,1)$ and $\beta\in[H,1)$. A centered Gaussian process $X=\{ X_t\}_{t\in[0,T]}$ with $X_0=0$ and covariance function $R$ belongs to the
\emph{class $\mathcal{X}^{(H,\beta)}$} if the following properties hold:
\begin{itemize}
\item[(1)] The incremental variance function $\sigma^2(s,t)$ defined by
$$
\sigma^2(s,t) = \E[(X_t - X_s)^2]
$$
is $C^1$ for $s\neq t$, and satisfies
$$
|\partial_s\partial_t \sigma^2(s,t)| \leq c|t-s|^{2H-2}.
$$
\item[(2)] The function $\sigma^2$ satisfies
$$
c|t-s|^{2\beta} \leq \sigma^2(s,t) \leq C|t-s|^{2H}.
$$
\end{itemize}
\end{defn}

For the given class $\mathcal{X}^{(H,\beta)}$, we obtain the following concentration inequalities for our purposes with relatively simple bounds. Rather surprisingly, the bounds depend on the value of $H$ and 
not on the parameter $\beta$. 
.
\begin{thm}\label{gaussian-concentration}
\label{concentration1}
Let $X\in\mathcal{X}^{(H,\beta)}$ and $\delta, N> 0$ be such that
$\delta N < T$. Define the $\R^{N}$-valued random variable $\mathbf{Y}=(Y_1, \cdots, Y_N)$ and the number $I$ by

\begin{equation*}
Y_k= X_{k\delta} - X_{(k-1)\delta} \quad \text{and} \quad  I^2 = \sum_{k=1}^{N} \E (X_{k\delta} - X_{(k-1)\delta})^2.
\end{equation*}

Then there exist constants $C_1,C_2$ independent of $\delta$,$N$ and $T$ such that for any $h >0$ the following bounds hold:

\begin{enumerate}
 \item The case when $H > \frac{1}{2}$, then

\begin{eqnarray*}
 \Pro \big( \big| \mathbf{\vert Y \vert}_2 - I \big| \ge h \big) \le C_1 \exp \Big( -C_{2} \frac{Nh^2}{(\delta N)^{2H}}\Big).
\end{eqnarray*}

 \item The case when $H \in (0,\frac{1}{2})$, then

 \begin{equation*}
\Pro \big( \big| \mathbf{\vert Y \vert}_2 - I \big| \ge h \big) \le C_1 \exp \Big( -C_{2} \frac{h^2}{\delta^{2H}}\Big).
\end{equation*}

\item The case when $H=\frac{1}{2}$, then
\begin{equation*}
\Pro \big( \big| \mathbf{\vert Y \vert}_2 - I \big| \ge h \big) \le C_1 \exp \Big( -C_{2} \frac{h^2}{\log N\delta^{2H}}\Big).
\end{equation*}
\end{enumerate}
\end{thm}
\begin{proof}
We begin by giving bound for the norm of covariance matrix of vector of the increments. Now we have
$$
\E[Y_iY_j] = \frac12\int_{(i-1)\delta}^{i\delta}\int_{(j-1)\delta}^{j\delta}\partial_s\partial_t \sigma^2(s,t)\ud s \ud t  
$$
which leads to bound 
\begin{equation}
\label{eq:cov_bound}
\big|  \E[Y_iY_j]   \big| \leq \frac12\delta^{2H}(1+|i-j|)^{2H-2}.
\end{equation}

Let now $\Gamma$ denote the covariance matrix of the Gaussian vector $(Y_1,\ldots, Y_N)$. Now a rough bound for norm $\Vert A\Vert_2$ for any matrix $A$ can be given as
$$
\Vert A\Vert_2 \leq \sqrt{\Vert A\Vert_1 \Vert A\Vert_{\infty}}.
$$
Furthermore, it is well-known that matrix norms $\Vert A\Vert_1$ and $\Vert A\Vert_{\infty}$ for a $n\times n$-square matrix are given by
$$
\Vert A\Vert_1 = \max_{1\leq j\leq n} \sum_{k=1}^n |a_{kj}|
$$
and
$$
\Vert A\Vert_{\infty} = \max_{1\leq j\leq n} \sum_{k=1}^n |a_{jk}|
$$
respectively. Hence, by symmetry of $\Gamma$, we get $\Vert \Gamma \Vert_2 \leq \Vert \Gamma \Vert_1$.
Consequently, by \eqref{eq:cov_bound}
$$
\Vert \Gamma \Vert_2 \leq C\max_{1\leq j\leq N} \sum_{k=1}^N \delta^{2H}(1+|k-j|)^{2H-2}.
$$
Now with some elementary computations we have
$$
\max_{1\leq j\leq N} \sum_{k=1}^N (1+|k-j|)^{2H-2} \leq C
$$
for $H<\frac12$ and 
$$
\max_{1\leq j\leq N} \sum_{k=1}^N (1+|k-j|)^{2H-2} \sim N^{2H-1}
$$
for $H>\frac12$ and 
$$
\max_{1\leq j\leq N} \sum_{k=1}^N (1+|k-j|)^{2H-2} \sim \log N
$$
for $H=\frac12$. This proves the claim.
\end{proof}

\begin{rmk}{\rm
Note that to obtain the concentration inequalities in Theorem \ref{gaussian-concentration} the only assumption used is the differentiability of the function $\sigma$ together with the bound 
$|\partial_s\partial_t \sigma^2(s,t)| \leq C|t-s|^{2H-2}$. However, this is only used to guarantee that the covariances of increments satisfy
$
\sum_{k=1}^n \E(Y_kY_j) \leq C\delta^{2H}N^{\gamma},
$
where $\gamma$ depends on the value of $H$.
Obviously one can derive similar concentration inequality for arbitrary Gaussian process which does not satisfy our assumptions.
}
\end{rmk}

\begin{rmk}{\rm
Another relevant bound to use is Frobenius norm given by $\Vert \Gamma \Vert_{F}^2 := \sum_{i,j=1}^n |\Gamma_{ij}|^2$ which was successfully used by Baudoin and Hairer \cite{b-h} for their purposes. However, 
Frobenius norm in general gives rough bound and while it provides exponential upper bounds for small deviations via our method, the Frobenius norm fails to provide optimal rates. Indeed, even in the simple 
case of a standard Brownian motion where the covariance matrix $\Gamma$ is a diagonal matrix, Frobenius norm gives half of the best possible rate. 
}\end{rmk}

With concentration inequalities derived in Theorem \ref{gaussian-concentration}, it is straightforward to give the following general theorem. We will only consider the cases $H\neq \frac12$. However, the 
case $H=\frac{1}{2}$ can be covered similarly.

\begin{thm}
\label{thm:gaussian}
Let $X\in\mathcal{X}^{(H,\beta)}$. Then there exist positive constants $C_1$ and $C_2$ such that for any $\epsilon\in(0,1)$ and for any interval $[0,T]$, we have

\begin{enumerate}
 \item The case when $H > \frac{1}{2}$:

\begin{equation*}
\begin{split}
\Pro( \Vert y \Vert_{\infty}\leq \epsilon) & \leq C_1\exp\left(-C_2 T^{2-2H}\epsilon^{-\gamma}\right)\\
&+\Pro\left( \Vert a \Vert_{\infty} \geq 2^{-2}T\epsilon^{\frac{\beta-2}{\beta}}\right),
\end{split}
\end{equation*}

where $\gamma= \frac{2-2\beta}{\beta}$.
\item The case when $H<\frac{1}{2}$:

\begin{equation*}
\begin{split}
\Pro(\Vert y \Vert_{\infty}\leq \epsilon) & \leq C_1\exp\left(-C_2 T\epsilon^{-\gamma}\right)\\
&+\Pro\left( \Vert a \Vert_{\infty} \geq 2^{-2}T\epsilon^{\frac{\beta-2}{\beta}}\right),
\end{split}
\end{equation*}
where $\gamma = \frac{2H-2\beta+1}{\beta}$. Notice that all constants are independent of $\epsilon$ and $T$.

\end{enumerate}
\end{thm}

\begin{proof}
Now by choosing
$$
I = \E\mathbf{\vert X \vert }_2
$$
we have the crucial lower bound $I \geq \sqrt{N}\delta^\beta$. Hence by selecting $N
\approx \frac{T}{\delta}$ and $\delta \approx c\epsilon^{\frac{1}{\beta}}$ we have $(N,\delta,I)\in\mathcal{A}_2$. Now the result follows directly by applying the concentration inequalities given in Theorem
\ref{gaussian-concentration}.
\end{proof}

\begin{rmk} {\rm 
It was pointed out in Li and Shao \cite{l-s} (see also Lifshits \cite{Lifshits}) that to obtain upper bound for small ball
probability it is not sufficient to have lower bound for incremental variance in general. With our method, this is exactly the
crucial element to obtain the exponential upper bounds.

}
 
\end{rmk}

\begin{exm}{\rm (\textit{small deviation for fractional Brownian motion in $\Vert \, \Vert_\infty$-norm})
A fractional Brownian motion $B^H=\{B^H_t\}_{t \in [0,T]}$ with Hurst parameter $H \in (0,1) $ is a centered continuous Gaussian process with covariance function
$$R_H(s,t)=\frac{1}{2} \{ s^{2H} + t^{2H} - \vert t -s \vert^{2H} \}.$$

Let $X=B^H$ be a fractional Brownian motion. Then the process $X$ satisfies in assumption $(2)$ of Definition \ref{defn:Xalpha} with $\beta=H$. When $H > \frac{1}{2}$, a direct application of Theorem \ref{thm:gaussian}
item $(1)$ yields the upper bound 

\begin{equation*}
\Pro \left( \Vert B^H \Vert_\infty \le \epsilon \right) \le C_1 \exp\{ - C_2 T^{2-2H} \epsilon^{- \frac{2-2H}{H}} \}, 
\end{equation*}
and when $H < \frac{1}{2}$, using item $(2)$, we obtain the upper bound 

\begin{equation*}
\Pro \left( \Vert B^H \Vert_\infty \le \epsilon \right) \le C_1 \exp\{ - C_2 T \epsilon^{- \frac{1}{H}} \}.
\end{equation*}
It is well known that (Monrad and Rootzen \cite{m-r} and Shao \cite{shao}) for fractional Brownian motion with Hurst index $H\in(0,1)$, we have 
\begin{equation}\label{optimal-fbm}
\bar{C}_1 \exp \left\{ - \bar{C}_2 T \epsilon^{- \frac{1}{H}} \right\} \le \Pro(\sup_{t\in[0,T]} |B_t^H| \le \epsilon) \leq C_1\exp \left\{ -C_2T\epsilon^{-\frac{1}{H}}\right\},
\end{equation}
provided that $\epsilon \leq T^{\frac{1}{H}}$. Note that the first inequality in $(\ref{optimal-fbm})$ is a consequence of the only assumption $\sigma^2(t,s) \le \vert t-s\vert^{2H}$, see for example 
\cite[Section 4.1]{l-s} and reference therein. Hence, when Hurst parameter $H<\frac{1}{2}$, we obtain the optimal rate of the small deviation for fractional Brownian motion $B^H$ using our general approach. A decisive 
reason we obtain the optimal rate in the case $H<\frac{1}{2}$ is that the spectral density function of the stationary incremental process of fractional Brownian motion is bounded in this case. This is topic of 
Theorem \ref{thm:gauss_stat_inc}.

}
\end{exm}

\begin{exm}{\rm (\textit{small deviation for fractional Brownian motion in $\beta$- H\"older norm $\Vert \, \Vert_\beta$})
Let $B^H$ be a fractional Brownian motion with Hurst parameter $H<\frac{1}{2}$ and fix $\beta \in (0,H)$. Then using Theorem \ref{general-thm3} and Theorem \ref{gaussian-concentration} item $(2)$, one can readily 
obtain the upper bound $\Pro (\Vert B^H \Vert_\beta \le \epsilon) \le C_1 \exp \{ - C_2 N \}$. Note that in the set $\widetilde{A}_2$, we have restriction $\sqrt{N} \delta^\beta \le I \epsilon^{-1}$, and together 
with estimate $I \geq \sqrt{N} \delta^H$, we deduce that with selections $\delta \approx \epsilon^{\frac{1}{H-\beta}}$ and $N \approx \delta^{-1}$, we have  

\begin{equation*}
\Pro (\Vert B^H \Vert_\beta \le \epsilon) \le C_1 \exp \{ - C_2 \epsilon^{- \frac{1}{H-\beta}} \}, 
\end{equation*}
for some constants $C_1, C_2 >0$. We recall that (see for example \cite{k-l-s-holdernorm} or \cite[Theorem 4.7]{l-s}) for fractional Brownian motion $B^H$ with Hurst parameter $H \in (0,1)$, and $\beta < H$, there are two 
constants $0< C_1 \le C_2 < \infty$ such that for all $\epsilon \in (0,1)$:

\begin{equation*}
\exp \{ - C_2 \epsilon^{- \frac{1}{H-\beta}} \} \le \Pro (\Vert B^H \Vert_\beta \le \epsilon ) \le \exp\{ - C_1 \epsilon^{- \frac{1}{H-\beta}} \}.
\end{equation*}
Hence with our general approach one can even gain the optimal rate of small deviation in H\"older norm for fractional Brownian motion with Hurst parameter $H<\frac{1}{2}$.

 }
\end{exm}

As a simple corollary, we obtain the following upper bounds for small ball probabilities when we have more information on the process $a$. We consider only the case $H<\frac{1}{2}$ since in this range, we 
can attain the optimal rate with our general theorem. However, exponential upper bound can be given for values $H\geq\frac{1}{2}$ using Theorem \ref{thm:gaussian}, item $(1)$.

\begin{cor}\label{cor1}
Let $X\in\mathcal{X}^{(H,\beta)}$ with $H<\frac{1}{2}$. Assume that the process $a$ is almost surely bounded. Then there exist positive constants $C_1$ and $C_2$ such that for any
$\epsilon\in(0,1)$ and for any interval $[0,T]$, we have
$$
\Pro( \Vert y \Vert_{\infty}\leq \epsilon) \leq C_1\exp\left(-C_2 T\epsilon^{-\frac{2H-2\beta+1}{\beta}}\right).
$$
The constants are independent of $\epsilon$ and $T$.
\end{cor}

Using a result by Marcus \& Sheep (see Lemma 3 in \cite{Lai-Stout}) we obtain the following exponential bound when the process $a$ is also Gaussian:

\begin{cor}
\label{cor2}
Let $X \in \mathcal{X}^{(H,\beta)}$ with $H<\frac{1}{2}$. Assume that $a$ is a Gaussian process such that $\Pro (\Vert a \Vert_{\infty} < \infty) >0$. Then there exist positive constants
$C_1$ and $C_2$ such that for any $\epsilon\in(0,1)$ and for any interval $[0,T]$, we have
$$
\Pro( \Vert y \Vert_{\infty} \leq \epsilon) \leq C_1\exp\left(-C_2 T\epsilon^{-\frac{2H-2\beta+1}{\beta}}\right).
$$
The constants are independent of $\epsilon$ and $T$.
\end{cor}

\begin{exm}{ \rm
Let $X=B^{H_1}$ and the process $a=B^{H_2}$ be fractional Brownian motions with Hurst parameter $H_1 <\frac12$ and $H_2 \in (0,1)$, and \underline{not necessarily the same}. Now Corollary \ref{cor2} implies that there exist positive constants
$C_1$ and $C_2$ such that for any $\epsilon\in(0,1)$ and for any interval $[0,T]$, we have
$$
\Pro( \Vert y \Vert_{\infty} \leq \epsilon) \leq C_1\exp\left(-C_2 T\epsilon^{-\frac{1}{H_1}}\right).
$$
We remark that up to our knowledge this is a new result giving
upper bound for small ball probabilities of the processes of the form $$ y_t = B^{H_1}_t + \int_{0}^{t} B^{H_2}_s \ud s.$$ We stress that $X=B^{H_1}$ and $a=B^{H_2}$ are not 
\underline{necessarily independent}. Notice that when two processes $X$ and $a$ are
independent, the \textit{Anderson's inequality} \cite[Theorem 2.13]{l-s} implies an upper bound

$$ \Pro ( \Vert y \Vert_{\infty} < \epsilon ) \le \Pro ( \Vert X \Vert_{\infty} < \epsilon ).$$

We also note that one can give exponential upper bounds also in the case $H_1 \geq \frac{1}{2}$ with a slightly different rates.
}
\end{exm}

In general the small deviation is studied for Gaussian processes with stationary increments. We end this section by providing the following result which links the small deviation problem for Gaussian 
processes with stationary increments into the theory of \textit{\textit{}Toeplitz matrices}, i.e.  matrices for which elements are constant along diagonals. Let $X$ be a Gaussian process with stationary 
increment with incremental variance function $\sigma^2(t,s):=\sigma^2(t-s)= \E( X_t - X_s)^2$. Let $\delta$ be a fixed number. For time points $\{t_k = k\delta T : \, k=0,1,\ldots\}$, we consider the 
associated stationary Gaussian noise as a sequence $Z=\{ Z_k\}_{k=0}^\infty$ defined by $Z_k = \frac{1}{\sigma(\delta)}\left(X_{t_k}-X_{t_{k-1}}\right)$. We will make the following assumption on the 
stationary sequence $Z$.

\begin{Assumption}
\label{assu:si}
Let $Z=\{ Z_k \}_{k=0}^\infty$ be defined as above. We assume that the stationary sequence $Z$ has a spectral density $f \in L^{\infty}[-\pi,\pi]$. 
\end{Assumption}

\begin{exm}{ \rm
Let $X$ be a fractional Brownian motion with Hurst index $H\in(0,1)$. Then by self--similarity of $X$, the sequence $Z$ satisfies $Z_k \stackrel{\text{law}}{=} T^{H}\tilde{Z}_k$, where $\tilde{Z}_k$ is a 
standard fractional Gaussian noise, i.e. $\tilde{Z}_k = B^H_k - B^H_{k-1}$. It is well-known that the sequence $Z$ has a spectral density which is bounded for values $H\leq \frac12$ and unbounded for values 
$H>\frac{1}{2}$ (see for example \cite[Proposition 2.1]{beran}).
}
\end{exm}

\begin{thm}
\label{thm:gauss_stat_inc}
Let $X$ be a Gaussian process with stationary increments and incremental variance function $\sigma^2(t-s)= \E(X_t - X_s)^2$ such that Assumption \ref{assu:si} is satisfied. Furthermore, assume that there exists a constant $\Delta$ such that $\sigma$ is invertible on $(0,\Delta)$, and that $\frac{\sigma(T\delta)}{\sigma(\delta)} \leq C$ for $\delta\in(0,\Delta)$. Then there exists constants $C_1$ and $C_2$ such that 
$$
\Pro( \Vert X \Vert_{\infty} \leq \epsilon) \leq C_1\exp\left(-C_2 \frac{T}{\sigma^{-1}(4\epsilon)}\right),
$$
where $\sigma^{-1}$ denotes the inverse of $\sigma$.
\end{thm}
\begin{proof}
Since $X$ has stationary increments, the covariance matrix of the sequence $\{ Z_k : \, k=1,\ldots,n\}$ is a Toeplitz matrix, say, $\Gamma_n$. Furthermore, the coefficients of this matrix are the Fourier 
coefficients of the spectral density function, say, $f$. Now it is well-known (see \cite[Theorem 1.1]{toeplitz}) that largest eigenvalue $\lambda^n_{\text{max}}$ of Toeplitz matrix $\Gamma_n$ as dimension 
$n$ of the matrix increases converges to the supremum of the spectral density function, i.e. 
$$
\lim_{n\to \infty}\lambda^n_{\text{max}} = \text{ess sup}|f(\lambda)| < \infty.
$$
Consequently, we obtain an upper bound $\Vert \Gamma\Vert_2 \leq \sigma^2(T\delta)$ from which the result follows immediately by applying Theorem \ref{general-thm} and Corollary \ref{cor:gaussian_general_con}.
\end{proof}

\begin{rmk}{\rm
Theorem \ref{thm:gauss_stat_inc} gives an easy to check condition to derive an upper bound for small deviations of Gaussian processes with stationary increments having bounded spectral density function. Also, 
it produces the optimal rate in the particular case of fractional Brownian motion with Hurst parameter $H\leq \frac12$. 
}\end{rmk}

\section{Application to stochastics integral representations}\label{application}

Given a process $X=\{X_t\}_{t \in [0,1]}$ with the natural filtration $\mathcal{F}=(\mathcal{F}_t)_{t \in [0,1]}$, it is an
interesting question that which random variables $\xi$, measurable with respect to the sigma-field $\mathcal{F}_1$, can be
represented as a stochastic integral
\begin{equation}\label{eq:integral_rep}
\xi = \int_0^1 \psi(s)\ud X_s
\end{equation}
for some adapted integrand $\psi(s)$. Especially, such questions are motivated by mathematical finance where the integral
representation $(\ref{eq:integral_rep})$ is interpreted as the \textit{hedging} of the \textit{contingent claim} $\xi$ by using financial strategy $\psi$. In order
to answer such problems, first one needs to define in which sense the stochastic integral exists, and therefore the definition of
the stochastic integral clearly depends on the integrator process $X$. This problem was studied for the standard Brownian motion by
Dudley \cite{d} who defined the integrals as the It\^o integral. Recently, the problem is explored to other integrator
processes taking into account the regularity of sample paths. In fact, the problem was considered for fractional
Brownian motion with the Hurst index $H>\frac{1}{2}$ by Mishura et al. \cite{m-s-v} where the authors proved that
the representation (\ref{eq:integral_rep}) holds if $\xi$ can be viewed as an end value of some $\vartheta$-H\"older process with any
$\vartheta>0$. Later on, their result was extended to general class of Gaussian processes by Viitasaari \cite{viitasaari}. The
results was further extended by Shevchenko and Viitasaari \cite{she-vii,she-vii2} to any integrator process $X$, not necessarily
Gaussian, which is H\"older continuous of order $\alpha>\frac{1}{2}$ and satisfies a small
ball estimate
\begin{equation}\label{smallball}
\Pro(\sup_{s\leq u \leq s+\Delta} |X_u- X_s| \leq \epsilon) \leq \text{exp}\left(-C\Delta\epsilon^{-\frac{1}{\alpha}}\right)
\end{equation}
for small enough $\Delta$.
 Note that the small ball estimate $(\ref{smallball})$ holds for many interesting Gaussian processes, in particular
for fractional Brownian motion. \\

Now, we apply our bounds for small deviations obtained in the Subsection \ref{gaussian} to integral representation problem. We restrict the analysis to Gaussian processes. However, it is straightforward to 
obtain modifications of the following results for general H\"older continuous processes by applying Theorem \ref{ex:bounded}. Let now $X$ be a Gaussian process. Now our results allows one to replace the small ball assumption
$(\ref{smallball})$ with more natural assumption; simply by assuming $X\in\mathcal{X}^{(H,\beta)}$ with some $H>\frac{1}{2}$
which in fact is drastically simple to check. This is the topic of the following theorem.

\begin{thm}
Let $X\in\mathcal{X}^{(H,\beta)}$ with $H>\frac{1}{2}$ and $\beta<\frac{3H}{H+2}$. Furthermore, assume that there exists an
$\mathcal{F}$-adapted process $\{z(t), t\ge 0\}$ having  H\"{o}lder continuous  paths of order
$\vartheta>\frac{2-2H}{1-\beta}\beta - 2H$ such that $z(1) = \xi$. Then there exists an $\mathcal{F}$-adapted process $\{\psi(t),
t\in[0,1]\}$  such that almost surely $\psi \in C[0,1)$ and
\begin{equation}\label{repres}
\int_0^1 \psi(s)\ud X_s = \xi \qquad \text{a.s.}
\end{equation}
where the stochastic integral is understood as a limit of Riemann-Stieltjes sums. 
\end{thm}

\begin{proof}
Let $\Delta_k$ be a sequence converging to zero such that $\sum_{k=1}^\infty \Delta_k = 1$. Consider the time points $t_n
= \sum_{k=1}^n \Delta_k$. Following arguments presented in \cite{she-vii2}, we obtain the result if we can choose the sequence
$\Delta_k$ and parameters $\mu$, $\gamma$, $\kappa$ and $\eta\in\left(1-H,\frac{1}{2}\right)$ in such way  that for small enough
$\epsilon$ and $\hat\epsilon$ the event
\begin{equation}
\label{event}
\left\{\sup_{t\in[t_{n-1},t_{n-1}+\Delta_n/2)} |X(t)- X(t_{n-1})| \leq \Delta_n^{\lambda-\hat\epsilon}\right\}
\end{equation}
happens only finite number of times, where $\lambda = \min(\mu+\vartheta,\gamma(H-\epsilon),\kappa)$. In addition, we have the following
three restrictions:
$$
\sum_{k=n}^{\infty}\Delta_k^{1-\eta-\mu}\rightarrow 0,  \quad
\sum_{k=n}^{\infty}\Delta_k^{2-\eta-\gamma}\rightarrow 0, \, \text{and} \,
\sum_{k=n}^{\infty}\Delta_k^{1+H-\eta-\mu-\kappa}\rightarrow 0.
$$

Moreover, we have to assume $\vartheta<H$. Now applying Theorem \ref{thm:gaussian} to the event (\ref{event}), we obtain
\begin{equation*}
\begin{split}
&\mathbb{P}\left(\sup_{t\in[t_{n-1},t_{n-1}+\Delta_n/2)} |X(t)- X(t_{n-1})| \leq \Delta_n^{\lambda-\hat\epsilon}\right) \\
& \qquad \leq C_1\exp\left(-C_2
\Delta_n^{2-2H-(\lambda-\hat\epsilon)\frac{2-2\beta}{\beta}}\right).
\end{split}
\end{equation*}
Therefore, using the Borel-Cantelli's Lemma, the event (\ref{event}) happens only finite number of times provided that $\Delta_n$
converges to zero fast enough and $2-2H-(\lambda-\hat\epsilon)\frac{2-2\beta}{\beta}
>0$. Combining with other three restrictions, we need to choose the parameters in such way that
\begin{itemize}
\item
$2-2H-(\lambda-\hat\epsilon)\frac{2-2\beta}{\beta} > 0$,
\item
$1-\eta-\mu>0$,
\item
$2-\eta-\kappa>0$,
\item
$1+H-\epsilon-\eta-\mu-\kappa > 0$,
\end{itemize}
and the result follows by choosing $\Delta_n$ such that it decays fast enough. First notice that by choosing
$\epsilon$ and $\hat\epsilon$ small enough, it is sufficient to have the following:\vspace*{0.5cm}\\
\begin{tabular}{llll}
& $(1)\quad \mu+\vartheta > \frac{1-H}{1-\beta}\beta$, & $(2) \quad \gamma H > \frac{1-H}{1-\beta}\beta$, & $(3) \quad \kappa > \frac{1-H}{1-\beta}\beta$, \\
& $(4)\quad 1-\eta-\mu>0$, & $(5)\quad 2-\eta-\kappa>0$, & $(6)\quad 1+H-\eta-\mu-\kappa > 0$.\vspace*{0.5cm}
\end{tabular}
Here (2) can be easily obtained by choosing $\gamma$ large enough. Next combining (1) and (4) we need
$\frac{1-H}{1-\beta}\beta -\vartheta < \mu < 1-\eta$ and together with $\eta\in\left(1-H,\frac{1}{2}\right)$ this is possible provided
that $1-H< 1+\vartheta-\frac{1-H}{1-\beta}\beta$ which leads to $\vartheta>\frac{1-H}{1-\beta}\beta - H$. Moreover, now we have to choose $\eta
\in\left(1-H, 1+\vartheta-\frac{1-H}{1-\beta}\beta\right)$. Furthermore, combining $\vartheta>\frac{1-H}{1-\beta}\beta - H$ with $\vartheta<H$ we end up to restriction $\beta < \frac{2H}{1+H}$.
Next combining restrictions (3) and (5) we obtain $\frac{1-H}{1-\beta}\beta < \kappa < 2-\eta$ which is again possible due to
previous choices. To conclude, we obtain (6) by choosing $\mu$ and $\kappa$ close their lower bounds provided that $\eta < 1+H + \vartheta - 2\frac{1-H}{1-\beta}\beta$. This is possible if $1-H
<  1+H + \vartheta - 2\frac{1-H}{1-\beta}\beta$ which leads to restriction $\vartheta>2\frac{1-H}{1-\beta}\beta-2H$ and together with $\vartheta<H$ this
yields the restriction $\beta < \frac{3H}{2+H}$.
\end{proof}

\begin{rmk}{\rm
Note that while we posed some restrictions for parameters $a$ and $\beta$, they are not very restrictive. For example, in
financial applications the random variable $\xi$ is usually some functional of the underlying
process $X$, and hence inherits the H\"older properties, i.e. $a$ can be taken arbitrary close to $H$. Similarly, for many cases
of interest the value $\beta$ is close to $H$ and certainly satisfies $\beta < \frac{3H}{H+2}$. }
\end{rmk}


\textbf{Acknowledgements}
The authors thank Giovanni Peccati for careful reading of the paper and many valuable comments. 
\bibliography{norris}

\bibliographystyle{plain}

\end{document}